\documentclass{amsart}

\usepackage{latexsym,amssymb}
\usepackage[english]{babel}

\usepackage{amsfonts, amsthm}
\usepackage{geometry}
\def\G{\Gamma}

\newtheorem{defi}{Definition}[section]
\newtheorem{prop}[defi]{Proposition}
\newtheorem{lem}[defi]{Lemma}
\newtheorem{rem}[defi]{Remark}

\newtheorem{cor}[defi]{Corollary}
\newtheorem{thm}[defi]{Theorem}

\newtheorem{conj}{Conjecture}
\theoremstyle{definition}
\newtheorem{ex}[defi]{Example}

\numberwithin{equation}{section}

\begin{document}

\title{The seating couple problem in even case}

\author{Mariusz Meszka}
\address{Faculty of Applied Mathematics,
AGH University of Krak\'ow,
al. Mickiewicza 30,
30-059 Krak\'ow, Poland}
\email{meszka@agh.edu.pl}

\author{Anita Pasotti}
\address{DICATAM - Sez. Matematica, Universit\`a degli Studi di Brescia,
Via Branze 43, I-25123 Brescia, Italy}
\email{anita.pasotti@unibs.it}

\author{Marco Antonio Pellegrini}
\address{Dipartimento di Matematica e Fisica, Universit\`a Cattolica del Sacro Cuore,
Via della Garzetta 48, I-25133 Brescia, Italy}
\email{marcoantonio.pellegrini@unicatt.it}

\subjclass[2010]{05C70, 05A17}
\keywords{seating couple problem, matching, Skolem sequence}

\begin{abstract}
In this paper we consider the seating couple problem with an even number of seats, which, using graph theory terminology, can be stated as follows.
Given a positive even integer $v=2n$ and a list $L$ containing $n$ positive integers not exceeding $n$,
is it always possible to find a perfect matching of $K_v$ whose list of edge-lengths is $L$?
Up to now a (non-constructive) solution is known only when all the edge-lengths are coprime with $v$.
In this paper we firstly present some necessary conditions for the existence of a solution.
Then, we give a complete constructive solution when the list consists of one or two distinct elements,
and when the list consists of consecutive integers $1,2,\ldots,x$, each one appearing with the same
multiplicity. Finally, we propose a conjecture and some open problems.
\end{abstract}

 \maketitle

\section{Introduction}

In \cite{PM} the authors considered the following problem proposed by Roland Bacher.
A king invites $n$ couples for dinner at his round table containing $2n+1$ seats, the king taking the last unoccupied chair.
The king has to address the following task: given an arbitrary set of $n$ couples, no one married for more than $n$ years,
is it always possible to seat all $n$ couples at his table according to the royal protocol stipulating that, if the two spouses of a couple are in their $i$-th year of marriage,
they have to occupy two chairs at \emph{circular distance} $i$ (where circular distance $i$ means that the two chairs are separated by exactly $i-1$ chairs)?
Using a mathematical language, the problem can be restated as follows:
given an arbitrary list of $n$ natural numbers $d_1,\ldots,d_n$ in $\{ 1,\ldots,n\}$, is it always possible to find an involution of
$2n+1$ circularly ordered points having a unique fixed point and consisting of $n$ disjoint transpositions exchanging respectively
two points at circular distance $d_1,\ldots,d_n$?

We point out that the same problem can be also stated using graph terminology (as already done in \cite{R}).
We prefer this choice because, in our proofs, we use tools from graph theory. To this purpose we introduce some definitions and notation,
see \cite{W} for a very good reference.
In this paper $K_v$ denotes the complete graph on $\{0,1,\ldots,v-1\}$ for any positive integer $v$.
The length $\ell(u, w)$ of an edge $\{u, w\}$ of $K_v$ is defined as
$$\ell(u, w) = \min(|u-w|, v-|u-w|).$$
If $\G$ is a subgraph of $K_v$, then the list of edge-lengths of $\G$ is the list $\ell(\G)$ of
the lengths (taken with their respective multiplicities) of all the edges of $\G$.
For convenience, if a list $L$ consists of $a_1$ $1$'s, $a_2$ $2$'s, $\ldots,  a_t$ $t$'s   one writes
$L=\{1^{a_1},2^{a_2},\ldots, t^{a_t} \}$, whose underlying set is the set of the elements $\{i : a_i>0\}$.

Given a graph $\G$ we denote by $V(\G)$ and $E(\G)$ its vertex-set and its edge-set, respectively.
If $\G$ has an odd (even, respectively) number of vertices, say $2n+1$ ($2n$, resp.), a \emph{near} $1$-\emph{factor}, resp. a $1$-\emph{factor}, $F$ of $\G$
is a set of $n$ pairwise disjoint edges of $\G$. In the even case, one also speaks of a \emph{perfect matching} of $\G$ and $V(F)=V(\G)$ holds.

Bacher proposed the following, see \cite{PM}.
\begin{conj}\label{Bacher}
There exists a solution to the king's problem if all distances, namely all edge-lengths, are invertible elements modulo the total number of seats, namely the number of the vertices of the complete graph.
\end{conj}

This conjecture has been generalized  in \cite{PP} as follows.
\begin{conj}\label{MPP}
Let $v=2n+1$ be an odd integer and  $L$ be a list of $n$ positive integers not
exceeding $n$.
Then there exists a near $1$-factor $F$ of $K_v$ such that $\ell(F)=L$ if and only if the following condition holds:
\begin{equation}\label{PP}
\left. \begin{array}{c}
\textrm{for any divisor $d$ of $v$, the number of multiples of $d$} \\
\textrm{appearing in $L$ does not exceed $\frac{v-d}{2}$.}
\end{array}\right.
\end{equation}
\end{conj}
With the aid of a computer, the authors of \cite{PP} verified its validity for all odd integers $v\leq 23$ and
pointed out that in the statement, the actual conjecture is the sufficiency. In
fact, they proved that condition \eqref{PP} is necessary. Furthermore, their main result is:

\begin{thm} \label{thm:MPP}
Let $L$ be a list of $n$ positive integers not exceeding $n$ and let $v = 2n + 1$.
Then, Conjecture \ref{MPP} holds whenever the underlying set $S$ of $L$ satisfies one of the following
conditions:
\begin{itemize}
\item[(1)] $|S|=1, 2$ or $n$;
\item[(2)] $S=\{1, 2, t\}$, where $t \geq 3$ is not coprime with $v$.
\end{itemize}
\end{thm}
Furthermore, Conjecture \ref{MPP} holds if $L=\{1^a, 2^b, t^c\}$ with $a + b \geq \lfloor\frac{t-1}{2}\rfloor$.
Finally, if $S=\{1, 2, t\}$ and $\gcd(v, t)= 1$, they also presented some partial results for the cases
not covered by Theorem \ref{thm:MPP}. In particular, they gave a complete solution for $t \leq 11$.

The case $n$ odd prime has been considered in \cite{PM}, in which the authors proved the existence of
a complete solution by a nonconstructive proof.

\begin{thm}
There exists a solution to the king's problem with $p$ seats for every odd prime~$p$.
\end{thm}

An alternative (not constructive) proof can be found in \cite{KP}, based on the famous Combinatorial Nullstellensatz \cite{A}.
Partial constructive results in the prime case can be found in \cite{R}.
We underline that Marco Buratti \cite{BM} proposed a strongly related conjecture, for details see \cite{OPPS} and the references therein.
Clearly, one can consider the same problem also in the even case.
In fact, in \cite{KSC} the authors proved  the following:

\begin{thm}\label{thm:KSC}
Let $v=2n$ be a positive integer and let $L$ be a list of $n$ integers coprime with $2n$ not exceeding $n$.
Then there exists a perfect matching $F$ of $K_v$ such that $\ell(F)=L$.
\end{thm}

We point out that this result was conjectured by Michal
Adamaszek  (private communication) and that, when $n$ is a prime, Theorem \ref{thm:KSC} was already proved by
Tam\'as R. Mezei in his Master's Thesis \cite{M}.
Clearly, the condition that all the edge-lengths are coprime with $2n$ is not necessary.
Take for instance $n=2$ and $L=\{2^2\}$; then, $F=\{\{0,2\},\{1,3\}\}$ is a perfect matching of $K_4$ such that $\ell(F)=L$.
This trivial example also shows that in the even case Condition  \eqref{PP} of Conjecture \ref{MPP} is not necessary.

In this paper we consider the seating couple problem in the even case removing the hypothesis that all the elements in the list are coprime
with the order of the complete graph.
Firstly, in Section \ref{sec:preliminary}, we present some necessary conditions for the existence of a perfect matching
of the complete graph having a given list of edge-lengths and we give results for some special classes of lists, including the case in which the underlying set has size one.

Then, it is natural to consider lists having exactly two distinct edge-lengths, see Section~\ref{sec:1x}. In this case, we obtain a full classification,
as described in the following.

\begin{thm}\label{thm:1x}
Let $n, x, y, a$ be four integers such that $1\leq x,y\leq n$, $x\neq y$  and $1\leq a<n$.
Let $d_x=\gcd(x,2n)$, $d_y=\gcd(y,2n)$ and  $d=\gcd(x,y,2n)$.
There exists a perfect matching $F$ of $K_{2n}$ such that $\ell(F)=\{ x^{n-a}, y^a\}$ if and only if
$d$ divides $n$ and one of the following cases occurs:
\begin{itemize}
 \item[(1)] $\frac{x}{d}$ is even, $\frac{y}{d}$ is odd, $n-a$ is even  and  either
 \begin{itemize}
 \item[(a)] $d_x$ divides $n$; or
 \item[(b)] $d_x$ does not divide $n$ and  $2a \geq d_x$;
\end{itemize}
 \item[(2)] $\frac{x}{d}$ is odd, $\frac{y}{d}$ is even, $a$ is even  and  either
 \begin{itemize}
 \item[(a)] $d_y$ divides $n$; or
 \item[(b)] $d_y$ does not divide $n$ and  $2(n-a) \geq d_y$;
\end{itemize}
 \item[(3)]   $\frac{x}{d}$ and $\frac{y}{d}$ are both odd, and the following two conditions  are both satisfied:
\begin{itemize}
\item[(a)] $a$ is even or $da\geq d_x$.
\item[(b)] $n-a$ is even or $d(n-a)\geq d_y$.
\end{itemize}
\end{itemize}
\end{thm}

In Section \ref{sec:246} we consider lists in which each element appears the same number of times.
In particular, we provide a complete solution for lists consisting of the integers
$1,2,\ldots,x$, for some $1\leq x\leq n$, each one appearing with the same  multiplicity.
We conclude our paper with some considerations and highlighting a conjecture and two open questions that, we believe, are of
particular interest for the seating couple problem.

\section{Necessary conditions and preliminary results}\label{sec:preliminary}

In the following given two integers $a$ and $b$ with $a\leq b$, by $[a,b]$ we mean the set with elements $a,a+1,\ldots,b$,
while it is empty when $a > b$.
Given an edge $\{u,w\}$ of $K_v$ it is useful to define
$\ell'(u,w)=|u-w|$. So, the length $\ell(u,w)$ is nothing but $\min(\ell'(u,w), v-\ell'(u,w))$.
Clearly, if $\ell'(u,w)\leq \left\lfloor \frac{v}{2}\right\rfloor$, then $\ell'(u,w)=\ell(u,w)$.
If $\G$ is a subgraph of $K_v$, $\ell'(\G)$ denotes the list $\{\ell'(e) : e\in E(\G)\}$.
Also, given a nonnegative integer $k$,  by $\G+k$ one means the graph with
vertex-set $\{u+k : u \in V(\G)\}$ and edge-set
$\{\{u+k, w+k\}: \{u,w\}\in E(\G)\}$. Note that $\G+k$ is not necessarily a subgraph of $K_v$.

\begin{rem}\label{rem:F1+F2}
Given a perfect matching $F_1$ of $K_{v_1}$ and a perfect matching $F_2$ of $K_{v_2}$, one can easily
get a perfect matching $F=F_1\cup (F_2+v_1)$ of $K_{v_1+v_2}$ such that $\ell'(F)=\ell'(F_1)\cup \ell'(F_2)$.
Note that, in general, the equality $\ell(F) =\ell(F_1)\cup \ell(F_2)$ does not have to hold.
\end{rem}

\begin{ex}
Consider for instance the perfect matchings $F_1=\{\{0,3\},\{1,2\}\}$ of $K_4$ and
$F_2=\{\{0,4\},\{1,2\},\{3,5\}\}$ of $K_6$. Then, $F=F_1\cup (F_2+4)=\{\{0,3\},\{1,2\},\{4,8\},\{5,6\},$ $\{7,9\}\}$
is a perfect matching of $K_{10}$ such that $\ell'(F)=\{1^2,2,3,4\}=\ell'(F_1)\cup\ell'(F_2)$.
On the other hand, $\ell(F)=\{1^2,2,3,4\}$,  while $\ell(F_1)\cup\ell(F_2)=\{1^3,2^2\}$.
\end{ex}

We now present some necessary conditions for the existence of a perfect matching with a given list of edge-lengths.

\begin{prop}\label{prop:MPP}
Let $v=2n$ be a positive integer and  $L$ be a list of $n$ positive integers not
exceeding $n$.
If there exists a perfect matching $F$ of $K_v$ such that $\ell(F)=L$, then
for any divisor $d$ of $v$ such that $d$ does not divide $n$, the number of multiples of $d$
appearing in $L$ does not exceed $\frac{v-d}{2}$.
\end{prop}
\begin{proof}
The proof is exactly the same as of \cite[Proposition 1]{PP}.
In this case the hypothesis $d$ does not divide $n$ ensures that $\frac{v}{d}$ is odd, which is
a necessary (and automatically satisfied) condition in  the proof of \cite[Proposition 1]{PP}.
\end{proof}

\begin{prop}\label{prop:EvenNumber}
Let $L=\{1^{a_1},2^{a_2},\ldots,n^{a_n}\}$ where $a_i\geq 0$ for every $i\in[1,n]$.
If there is a perfect matching $F$ of $K_{2n}$ such that $\ell(F)=L$, then
$\sum\limits_{i=1}^{\left\lfloor\frac{n}{2}\right\rfloor} a_{2i}$ is even.
\end{prop}

\begin{proof}
Let $L_1$ be the sublist of $L$ containing exactly all the even elements of $L$, and set $L_2=L\setminus L_1$.
Hence, $L_2$ contains all the odd elements of $L$. Let $F_1$ and $F_2$ be the subgraphs of $F$  such that $\ell(F_1)=L_1$
and $\ell(F_2)=L_2$. Clearly, the length of an edge is odd if and only if its end-vertices have different parity.
Hence, $V(F_2)$ contains the same number (that is $|L_2|$) of even and odd numbers.
This implies that also $V(F_1)=V(K_{2n})\setminus V(F_2)$ contains the same number of even and odd numbers.
Since the end-vertices of the edges of $F_1$ have the same parity, this implies that $|L_1|$ is even.
The statement follows.
\end{proof}

\begin{prop}\label{prop:b}
Let $c$ and $n$ be positive integers and let $F_0,F_1,\ldots,F_{c-1}$ be perfect matchings of $K_{2n}$
such that $\ell(F_i)=\{1^{a_{i,1}},2^{a_{i,2}},\ldots, n^{a_{i,n}}\}$, where $a_{i,j}\geq 0$.
Then, there exists a perfect matching $F$ of $K_{2nc}$ such that $\ell(F)=\{c^{b_1},(2c)^{b_2},\ldots, (nc)^{b_n}\}$,
where $b_j=\sum\limits_{i=0}^{c-1}a_{i,j}$.
\end{prop}

\begin{proof}
Let $R_i$ be the matching of $K_{2nc}$ obtained from $F_i$ applying the relabeling $u \mapsto cu+i$.
Hence, $\ell(R_i)=\{c^{a_{i,1}},(2c)^{a_{i,2}},\ldots, (nc)^{a_{i,n}}\}$ and
$V(R_i)=\{w : w\in[0,2nc-1] \textrm{ and } w\equiv i\pmod{c} \}$.
We conclude that $R_0\cup R_1\cup \ldots \cup R_{c-1}$
is a perfect matching of $K_{2n}$ with the required properties.
\end{proof}

\begin{prop}\label{prop:c}
Let $c$ and $n$ be positive integers and let $F$ be a perfect matching of $K_{2nc}$ such that
$\ell(F)=\{c^{b_1},(2c)^{b_2},\ldots, (nc)^{b_n}\}$, where $b_j\geq 0$.
Then, there exist $c$ (not necessarily distinct) perfect matchings $F_0,F_1,\ldots,F_{c-1}$  of $K_{2n}$
such that $\ell(F_i)=\{1^{a_{i,1}},2^{a_{i,2}},\ldots, n^{a_{i,n}}\}$, where $a_{i,j}\geq 0$ and
$b_j=\sum\limits_{i=0}^{c-1}a_{i,j}$.
\end{prop}

\begin{proof}
The end-vertices of each edge of $F$ belong to the same congruence class modulo $c$.
So, considering the vertices in each congruence class, we obtain $c$ submatchings
$S_0,S_1,\ldots,S_{c-1}$ of $F$ such that $V(S_i)=\{u : u\in[0,2nc-1] \textrm{ and } u\equiv i\pmod{c} \}$
and $\ell(S_i)=\{c^{a_{i,1}},(2c)^{a_{i,2}},\ldots,$ $(nc)^{a_{i,n}}\}$.
Applying the relabeling $u\mapsto \frac{u-i}{c}$ we obtain the perfect matchings $F_0,F_1,\ldots,F_{c-1}$  of $K_{2n}$.
\end{proof}

Previous proposition is a  useful tool for getting some non-existence results as shown in the following example.

\begin{ex}
Applying Proposition \ref{prop:c} one can see that there is no perfect matching $F$ of $K_{20}$
such that $\ell(F)=\{4^3,6^7\}$. In fact, if we take $c=2$ the existence of $F$ would imply the
existence of two perfect matchings $F_0,F_1$ of $K_{10}$ such that $\ell(F_0)=\{2^{a_{0,2}},3^{a_{0,3}}\}$
and $\ell(F_1)=\{2^{a_{1,2}},3^{a_{1,3}}\}$, where $a_{0,2}+a_{1,2}=3$.
By Proposition \ref{prop:EvenNumber} we have a contradiction.
\end{ex}

We now consider the existence of perfect matchings for some special lists.

\begin{prop}\label{prop:xn}
Let $x$ and $n$ be two positive integers such that $1\leq x \leq n$.
There exists a perfect matching $F$ of $K_{2n}$ such that $\ell(F)=\{x^n\}$
if and only if $\gcd(x,2n)$ is a divisor of $n$.
\end{prop}

\begin{proof}
Set $d=\gcd(x,2n)$. If $d$ does not divide $n$, the non-existence of the perfect matching follows by Proposition \ref{prop:MPP}.
Suppose now that $d$ divides $n$. Let
$$F=\left\{\left\{2ix+j,(2i+1)x+j\right\}: i\in\left[0,\frac{n}{d}-1\right], j \in\left[0,d-1\right]\right\},$$
where the elements are considered modulo $2n$.
Clearly, $F$ is a set of $n$ edges, each of length $x$.
It is not hard to see that $V(F)=V(K_{2n})$, namely that $F$ is a perfect matching of $K_{2n}$ such that $\ell(F)=\{x^n\}$.
\end{proof}

\begin{ex}
Take $x=9$ and $n=12$, hence $d=\gcd(x,2n)=3$. In particular $d$ divides $n$, so there exists a perfect matching
$F$ of $K_{24}$ such that $\ell(F)=\{9^{12}\}$. Following the proof of previous proposition, we have
$F=\left\{\{18i+j,18i+9+j\}: i\in [0,3], j \in [0,2]\right\}$,
that is
$$\begin{array}{rcl}
F &=& \{\{0,9\},\{18,3\},\{12,21\},\{6,15\},\{1,10\},\{19,4\},\{13,22\},\{7,16\},\\
 &&\{2,11\},\{20,5\},\{14,23\},\{8,17\}\}.
\end{array}$$
\end{ex}

\begin{prop}
Let $x,y,n$ be three integers such that $1\leq x< y <n$.
There is no perfect matching $F$ of $K_{2n}$ such that $\ell(F)=\{x,y,n^{n-2}\}$.
\end{prop}
\begin{proof}
Let $L=\{x,y,n^{n-2}\}$ with $x,y$ and $n$ as in the statement.
For the sake of contradiction, suppose that there exists a perfect matching $F$ of $K_{2n}$ such that $\ell(F)=L$.
This implies that there is a subgraph $F'$ of $F$ such that $\ell(F')=\{n^{n-2}\}$.
Clearly, the edges of $F'$ are  $\{i,n+i\}$ for $i\in[0,n-1]\setminus \{u,w\}$ for some $u,w$.
In other words, $V(F')=V(K_{2n})\setminus\{u,n+u,w,n+w\}$ for some $u,w\in[0,n-1]$.
It is easy to see that it is not possible to match the vertices $u,n+u,w,n+w$ in such a way to have two disjoint edges
with distinct lengths $x$ and $y$ both different from $n$.
\end{proof}

To conclude this section we propose a constructive asymptotic result.
\begin{prop}
Let $t$ and $n$ be integers with $1\leq t\leq n$ and
let $L=\{1^{a_1}, 2^{a_2}, \ldots, t^{a_t}\}$ with $a_i\geq 0$ and $|L|=n$.
Set $R=\sum\limits_{j\geq 1} a_{2 j}$ and
$S=\sum\limits_{j\geq 2} \left\lfloor\frac{j-1}{2}\right\rfloor a_j$.
There exists a perfect matching $F$ of $K_{2n}$ such that $\ell(F)=L$ whenever $R$ is even and $a_1\geq S$.
\end{prop}
\begin{proof}
For every $x,y\geq 0$, define
$F_x=\{ \{0,2x+1\}\} \cup \{ \{ 2i+1,2i+2 \}: i \in [0,x-1] \}$
and
$F_{x,y}=\{ \{0,2x+2\}, \{2x+1, 2x+2y+3\}\} \cup \{ \{2i+1, 2i+2\} : i \in [0,x-1]\cup [x+1,x+y] \}$.
Then, $F_x$ is a perfect matching of $K_{2(x+1)}$ such that $\ell'(F_x)=\{1^x, 2x+1\}$,
while $F_{x,y}$ is a perfect matching of $K_{2(x+y+2)}$ such that $\ell'(F_{x,y})=\{1^{x+y},2x+2,2y+2\}$.
The statement follows from Remark \ref{rem:F1+F2}.
\end{proof}

\begin{ex}
Let $L=\{1^{21},2^7,4,5^2,10^4\}$. Hence, $|L|=35$ and so we are working in $K_{70}$.
Note that all the conditions of previous proposition are satisfied, since $R=12$ is even and
and $a_1\geq S= 21$.
A perfect matching of $K_{70}$ with the required properties is
$$F_{4,4} \cup (F_{4,4}+20) \cup (F_2+40)\cup (F_2+46)\cup (F_{0,1}+52) \cup (F_{0,0}+58)\cup (F_{0,0}+62)
\cup (F_{0,0}+66).$$
\end{ex}

\section{A complete solution for the two edge-lengths case}\label{sec:1x}

In this section we present the necessary and sufficient conditions for the existence
of a perfect matching $F$ of $K_{2n}$ such that $\ell(F)=\{x^{n-a}, y^a\}$ for any $a\in [1,n-1]$
and any $x,y \in [1, n]$ with $x\neq y$.
Note that the cases $x=y$ and $a=n$ have been already  solved in Proposition \ref{prop:xn}.
In other words, we prove Theorem \ref{thm:1x} and we begin with the case $\gcd(x,y,2n)=1$, where $x$ is even.

\begin{thm}\label{thm:xy}
Let $n,a,x,y$ be integers such that $x$ is even, $x \neq y$, $1 \leq x,y\leq  n$   and $1\leq a <n$.
Let $d=\gcd(x,2n)$ and suppose that $\gcd(d,y)=1$.
There exists a perfect matching $F$ of $K_{2n}$ such that $\ell(F)=\{ x^{n-a}, y^{a} \}$
if and only if one the following cases occurs:
\begin{itemize}
 \item[(1)] $d$ divides $n$ and $n-a$ is even;
 \item[(2)] $d$ does not divide $n$ and $n-a$ is an even integer such that $n-a\leq \frac{2n-d}{2}$.
\end{itemize}
\end{thm}

\begin{proof}
Since $x$ is even, $d$ is also even: by Proposition \ref{prop:EvenNumber}  we may assume that $n-a$  is even.
Furthermore, notice that the integers $ix+jy$, with $i \in [0,\frac{2n}{d}-1]$ and $j\in [0,d-1]$ are pairwise distinct modulo $2n$.
In fact, suppose $i_1x +j_1y \equiv i_2 x+ j_2 y \pmod{2n}$ for some $i_1,i_2 \in [0, \frac{2n}{d}-1]$ and some $j_1,j_2 \in [0,d-1]$.
Then, $j_1 y \equiv j_2 y \pmod{d}$, whence $j_1 \equiv j_2 \pmod{d}$ as $\gcd(d,y)=1$.
We obtain $j_1=j_2$ and so $i_1x  \equiv i_2 x \pmod{2n}$. It follows that $i_1\equiv i_2 \pmod{\frac{2n}{d}}$,
whence $i_1=i_2$.

If $d$ divides $n$, set $\bar n =\frac{n}{d}$; otherwise,  set $\bar n =\frac{2n-d}{2d}$  since
$\frac{2n}{d}$ is an odd integer.
Let $q,r$ be two integers such that $n-a= 2 \bar n q  + 2r$ where $0\leq r < \bar n$.
Working modulo $2n$, take the three matchings:
$$\begin{array}{rcl}
A &= &\left\{ \{2ix, (2i+1)x \}, \{2ix+y, (2i+1)x+y\} : i \in [0,r-1] \right\} \cup \\
&& \left\{\{ ix, ix+y\}: i \in [2r,2 \bar n -1]\right\},\\[2pt]
B &= & \left\{\{ 2ix, (2i+1)x \},\{ 2ix+y, (2i+1)x+y \} : i \in [0, \bar n -1] \right\},\\[2pt]
C & =& \left\{ \{ix, ix+y \}: i \in [0,2\bar n -1] \right\}.
  \end{array}$$
Notice that $V(A)=V(B)=V(C)=\{ix, ix+y: i \in [0,2\bar n -1] \}$.
Furthermore, $\ell(A)=\{ x^{2r},y^{2(\bar n-r)} \}$,
$\ell(B)=\{x^{2\bar n} \}$ and
$\ell(C)=\{y^{2\bar n} \}$.

If we are in the case (1), take
$$F=A \cup \bigcup_{k=1}^{q} (B + 2ky)\cup \bigcup_{k=q+1}^{\frac{d-2}{2}} (C+2ky).$$
In this way  we get $2r+2q\bar n = n-a$ edges of length $x$ and
$2(\bar n -r)+2\left(\frac{d}{2}-q-1\right)\bar n=a$ edges of length $y$ (note that $q<\frac{d}{2}$).
From the previous argument it follows that $V(F)=[0,2n-1]$.
We conclude that $F$ is a perfect matching of $K_{2n}$ such that $\ell(F)=\{x^{n-a}, y^{a}\}$.

In case (2), Proposition \ref{prop:MPP} implies $n-a\leq \frac{2n-d}{2}$.
Define the additional matching
$$D=\left\{\left\{\left(\frac{2n}{d}-1\right) x + 2jy, \left(\frac{2n}{d}-1\right)x+(2j+1)y\right\}:  j \in \left[0,\frac{d-2}{2}\right]\right\}.$$
We have $V(D)=\left\{\left(\frac{2n}{d}-1\right)x+jy: j \in [0,d-1] \right\}$ and $\ell(D)=\left\{y^{\frac{d}{2}}\right\}$.

If $r=0$, take
$$F=\bigcup_{k=0}^{q-1} (B + 2ky)\cup \bigcup_{k=q}^{\frac{d-2}{2}} (C+2ky)\cup D.$$
In this way  we get $2q\bar n = n-a$ edges of length $x$ and
$2\left(\frac{d}{2}-q\right)\bar n+ \frac{d}{2}= \frac{d}{2}(1+2\bar n)-(n-a)= a$ edges of length $y$
(note that $q\leq \frac{d}{2}\leq a$).
If $r>0$, take
$$F=A \cup \bigcup_{k=1}^{q} (B + 2ky)\cup \bigcup_{k=q+1}^{\frac{d-2}{2}} (C+2ky)\cup D.$$
In this way  we get $2r+2q\bar n = n-a$ edges of length $x$ and
$2(\bar n -r)+2\left(\frac{d}{2}-q-1\right)\bar n+ \frac{d}{2}= \frac{d}{2}(1+2\bar n)-(n-a)= a$ edges of length $y$
(note that $q+1\leq \frac{d}{2}\leq a$).
Since $2\bar n-1=\frac{2n}{d}-2$, in  both cases we have $V(F)=[0,2n-1]$.
We conclude that $F$ is a perfect matching of $K_{2n}$ such that $\ell(F)=\{x^{n-a}, y^{a}\}$.
\end{proof}

\begin{ex}
To obtain a perfect matching $F$ of $K_{40}$ such that $\ell(F)=\left\{5^6, 12^{14}\right\}$, write
$x=12$, so $d=\gcd(12,40)=4$ divides $n=20$, hence we are in Case (1) of previous theorem.
Following the notation of the proof, write  $n-a= 14 = 2\cdot 5 \cdot 1 + 2\cdot 2$  and take $F=A\cup(B+10)$ where
$$\begin{array}{rcl}
A & =&  \{\{24i,24i+12\}, \{24i+5,24i+17\}: i \in[0,1]\} \cup \{ \{12i,12i+5\} : i \in [4,9]\}, \\
B &=& \{\{24i, 24i+12\}, \{24i+5, 24i+17\}: i \in [0,4]\}.
    \end{array}$$
Hence,
$$\begin{array}{rcl}
F & =&  \{\{0,12\}, \{24,36\},  \{5,17\}, \{29,1\},  \{8, 13\}, \{20,25\}, \{32,37\}, \{4,9\},\{16,21\}, \\
&& \{28,33\} \}\cup  \{ \{10,22\}, \{34,6\}, \{18,30\}, \{2,14\},\{26,38\}, \{15, 27\}, \{39, 11\}, \\
&& \{23, 35\}, \{7, 19\}, \{31, 3\}  \}.
    \end{array}$$
Note that in this case $C$ is empty.
\end{ex}

\begin{ex}
To obtain a perfect matching $F$ of $K_{42}$ such that $\ell(F)=\left\{7^{13}, 12^{8}\right\}$, write
$x=12$. Since $d=\gcd(12,42)=6$ does not divide $n=21$,  we are in Case (2) of Theorem \ref{thm:xy}.
Following the notation of the proof, write
$n-a= 8 = 2\cdot 3 \cdot 1 + 2\cdot 1$ and take $F=A \cup (B+14)\cup (C+28) \cup D$ where
$$\begin{array}{rcl}
A & =&  \{\{24i,24i+12\}, \{24i+7,24i+19\}: i=0\} \cup \{ \{12i,12i+7\} : i \in [2,5]\}, \\
B &=& \{\{24i, 24i+12\}, \{24i+7, 24i+19\}: i \in [0,2]\},\\
C &=& \{\{12i, 12i+7\}: i \in [0,5]\},\\
D &=& \{\{14j+30, 14j+37\}: j \in [0,2]\}.
    \end{array}$$
That is, take
$$\begin{array}{rcl}
F &= & \{  \{ 0,12\},\{ 7,19\}, \{24,31\},\{36,1\},\{6,13\}, \{18,25\}\}\cup
\{ \{14, 26\}, \{38, 8\}, \\
&& \{20, 32\},\{ 21, 33\}, \{3, 15\}, \{27, 39\} \}\cup
\{ \{28,35\}, \{40,5\}, \{10,17\}, \{22,29\}, \\
&& \{34,41\},\{4,11\} \} \cup \{\{30,37\},\{2,9\},\{16,23\}\}.
\end{array}$$
\end{ex}

We now consider the case when $x$ is odd.

\begin{lem}\label{large_a}
Let $n,x,a$ be three integers such that $x$ is odd, $1< x< n$ and $\frac{n}{2}\leq a<n$.
There exists a perfect matching $F$ of $K_{2n}$ such that $\ell(F)=\{1^a, x^{n-a}\}$.
\end{lem}

\begin{proof}
Let $b=n-a$, then $a\geq b>0$ and $b\leq \frac{n}{2}$. Let $s$ and $t$ be nonnegative integers such that $b=sx+t$,
where $0\leq t< x$. Consider separately three cases.

\noindent Case I: $t=0$.\\
Let $A=\{\{2ix+j,(2i+1)x+j\}: i\in [0,s-1],\; j\in [0,x-1]\}$. Note that $A$ is a matching containing $b$ edges
of length $x$.
Vertices which are not in $V(A)$ make the interval
$H=[2sx,2n-1]$ of cardinality $2a$, from which we can obviously construct a matching $B$ with $a$ edges of length $1$.
Thus we can take $F=A\cup B$.

\noindent Case II: $t$ is odd.\\
Define two vertex disjoint matchings:
$$\begin{array}{rcl}
A &=& \{\{2ix+j,(2i+1)x+j\}: i\in [0,s-1],\; j\in [0,x-1]\},\\[2pt]
B &=& \{\{2sx+j,(2s+1)x+j\}: j\in [0,t-1]\}.
\end{array}$$
Notice that  $A\cup B$ contains $b$ edges of length $x$.
Vertices which are not in $V(A\cup B)$ make two disjoint intervals:
$H=[2sx+t,2sx+x-1]$ and $I=[2sx+x+t,2n-1]$.
Both $H$ and $I$ are nonempty and each contains an even number of consecutive integers, so it is immediate to construct
a matching $C$ containing $\frac{|H|+|I|}{2}=a$ edges of length $1$. Then, $F=A\cup B\cup C$ is a
perfect  matching of $K_{2n}$ such that $\ell(F)=\{1^a, x^{n-a}\}$.

\noindent Case III: $t$ is even and $t\geq 2$.\\
Let $z$ be a positive integer such that $n=x+z$, and write $k=\left\lfloor\frac{x}{2z}\right\rfloor$.
Consider two subcases.

\noindent Subcase III.A: $s>0$ or $k=0$.\\
Similarly to the above cases, define the following three vertex disjoint matchings:
$$\begin{array}{rcl}
A &=& \{\{2ix+j,(2i+1)x+j\}:  i\in [0,s-1],\; j\in [0,x-1]\},\\[2pt]
B &=& \{\{2sx,(2s+1)x\}\},\\[2pt]
C &=& \{\{(2s+1)x+j,(2s+2)x+j\}:  j\in [1,t-1]\}.
\end{array}$$
Note that $A\cup B\cup C$ is a set of $b$ edges, each of length $x$, since $2z+x\equiv -x\pmod{2n}$.
Vertices which are not in $V(A\cup B\cup C)$ make three disjoint intervals:
$H=[2sx+1,2sx+x-1]$, $I=[2sx+x+t,2sx+2x]$ and $J=[2sx+2x+t,2n-1]$.
It is trivial that both $H$ and $I$ are nonempty. Now we show that $2sx+2x+t<2n-1$.
If $s>0$, then $b=sx+t\geq x+t$. On the other hand $b\leq \frac{n}{2}$, whence $x\leq sx\leq \frac{n}{2}-t$.
So,
$$2sx+2x+t=(sx+t)+sx+2x=b+sx+2x\leq \frac{n}{2}+\left(\frac{n}{2}-t\right)+2\left(\frac{n}{2}-t\right)< 2n-1.$$
Suppose now $s=k=0$. Since $s=0$, we have $b=t<x$; since $k=0$, we get $x<2(n-x)$ and hence $x<\frac{2n}{3}$.
So,
$$2sx+2x+t=2x+t\leq 2x+x-1<2n-1.$$
Hence, we have proved that $J$ is nonempty too.
Each of $H$, $I$ and $J$ has even cardinality, so we can get
a matching $D$ containing $\frac{|H|+|I|+|J|}{2}=a$ edges of length 1. Then it is sufficient to take $F=A\cup B\cup C\cup D$.

\noindent Subcase III.B: $s=0$ and $k>0$.\\
Thus $x>2z$ and, since $n=x+z$, then $2n<3x$. Moreover, $b=t$ is even and $b\leq \frac{n}{2} < \frac{3x}{4}$.
Let $b=2pz+r$, where $0\leq r<2z$ and $p\geq 0$. Then $r$ is even.
Define the following vertex disjoint matchings:
$$\begin{array}{rcl}
A &=&\{\{2iz,2(i+1)z+x\}:  i\in [0,p-1]\},\\[2pt]
B &=& \{\{2iz+j,2iz+x+j\}:  i\in [0,p-1],\; j\in [1,2z-1]\},\\[2pt]
C &=& \{\{2pz+j,2pz+x+j\}:  j\in [1,r-1]\},
\end{array}$$
and
$$D = \{\{2pz,2(p+1)z+x\}\}$$
if $r>0$, $D=\emptyset$ otherwise.
The set $A\cup B\cup C\cup D$ contains $b$ edges, each of length $x$.
Notice that vertices which are not in $V(A\cup B\cup C\cup D) $ make three disjoint intervals:
$H=[b,x]$, $I=[x+b-h+1,2(p+1)z+x-h]$ and $J=[2(p+1)z+x+1,2n-1]$, where $h=0$ if $r=0$ and $h=1$ otherwise.
Each of these intervals contains an even number of consecutive integers, $H$ and $I$ are nonempty while $J$ is empty only
when $a=b=x-1=2$. Thus,
a matching $E$ containing $\frac{|H|+|I|+|J|}{2}=a$ edges of length 1 can be easily constructed.
In this case, take $F=A\cup B\cup C\cup D\cup E$.
\end{proof}

\begin{ex}
Consider the list $L=\{1^5,7^4\}$. Following the notation of the proof of the previous lemma we get:
$x=7$, $n=9$, $a=5$, $b=4$, hence $s=0$ and $t=4$. This means that to construct $F$ Case III is applied.
Then $z=2$ and $k=1$.  By Subcase III.B, $p=1$, $r=0$ and
$A=\{\{0,11\}\}$, $B=\{\{1,8\},\{2,9\},\{3,10\}\}$, $C=D=\emptyset$, $H=[4,7]$, $I=[12,15]$, $J=[16,17]$.
Thus $F=A\cup B\cup E$, where $E=\{\{4,5\},\{6,7\},\{12,13\},\{14,15\},\{16,17\}\}$.
\end{ex}

\begin{prop} \label{small_a}
Let $n,x,a$ be three integers such that $1< x< n$, $\gcd(x,2n)=1$ and $1\leq a<n$.
There exists a perfect matching $F$ of $K_{2n}$ such that $\ell(F)=\{1^a, x^{n-a}\}$.
\end{prop}

\begin{proof}
By Lemma \ref{large_a}, it remains to consider the case when $1\leq a<n-a$.
Since $\gcd(x,2n)=1$,  there exists an integer $y$ such that $1<y<2n$ and $xy\equiv 1\pmod{2n}$.
By Lemma \ref{large_a}, $K_{2n}$ contains a perfect matching $F'$ such that $\ell(F')=\{1^{n-a}, y^{a}\}$
if $y\leq n$, or $\ell(F')=\{1^{n-a}, (2n-y)^{a}\}$ if $n <y<2n$. In both cases,
we apply the relabeling $i\mapsto xi$ to all vertices of $F'$ to get 
a perfect matching $F$ of $K_{2n}$ such that $\ell(F)=\{1^a, x^{n-a}\}$.
\end{proof}

\begin{lem}\label{large_an}
Let $n$ and $a$ be two integers such that $n$ is odd, $a$ is even  and $2\leq a<n$.
There exists a perfect matching $F$ of $K_{2n}$ such that $\ell(F)=\{1^a, n^{n-a}\}$.
\end{lem}

\begin{proof}
Define the matching
$A = \{\{ j, n+j\}: j\in [0,n-a-1]\}.$
Clearly,  $A$ contains $n-a$ edges of length $n$. Vertices which are not in $V(A)$ make two disjoint intervals:
$H=[n-a, n-1]$ and $I=[2n-a, 2n-1]$.
Both $H$ and $I$ contain $a$ consecutive integers, so it is immediate to construct
a matching $B$ consisting of $a$ edges of length $1$. Then, $F=A\cup B$ is 
a perfect matching $F$ of $K_{2n}$ such that $\ell(F)=\{1^a, n^{n-a}\}$.
\end{proof}

The following result extends \cite[Theorem 2.2]{KSC} by replacing the hypothesis $\gcd(x_i,2n)=1$
with the weaker hypothesis $\gcd(x_i,2)=1$.

\begin{lem} \label{necessary}
Let $L$ be a list of length $n$ such that each its element $x_i$ is an odd positive integer and does not exceed $n$.
If $K_{2n}$ contains a perfect matching $F$ such that $\ell(F)=L$ then there exist $\varepsilon_1,\varepsilon_2,\ldots,\varepsilon_n$
such that $\varepsilon_i\in \{-1,1\}$ and $\sum\limits_{i=1}^{n} \varepsilon_i x_i \equiv n\pmod {2n}$.
\end{lem}

\begin{proof}
Each edge of $F$ has odd length $x_i$ so one its end-vertex, $u_i$, has odd label and the other, $w_i$, even.
Clearly, $u_i-w_i\equiv \varepsilon_i x_i \pmod{2n}$ for some $\varepsilon_i=\pm1$.
Set $s=\sum\limits_{i=1}^{n}  \varepsilon_i x_i $. Since the sum of all odd labels of vertices in $K_{2n}$ is equal to $n^2$ and
the sum of all vertices with even labels is $n^2-n$, we obtain $s\equiv n\pmod {2n}$.
\end{proof}

\begin{thm}\label{thm:x odd}
Let $n,x,y,a$ be four integers such that $x$ and $y$ are odd, $x\neq y$, $1\leq x,y\leq n$ and $1\leq a<n$. Let $d_x=\gcd(x,2n)$
and $d_y=\gcd(y,2n)$, and suppose that $\gcd(d_x,d_y)=1$.
There exists a perfect matching $F$ of $K_{2n}$ such that $\ell(F)=\{ x^{n-a}, y^a\}$ if and only if both conditions are satisfied:
\begin{itemize}
\item[(1)] $a$ is even or $a\geq d_x$;
\item[(2)] $n-a$ is even or $n-a\geq d_y$.
\end{itemize}
\end{thm}

\begin{proof}
To prove the necessity, suppose to the contrary that (1) or (2) is not satisfied. First consider the case when
$a$ is odd and $1\leq a<d_x$.
By Lemma \ref{necessary}, if $K_{2n}$ contains a perfect matching $F$ such that $\ell(F)=
\{ x^{n-a}, y^a\}$, then  there exist $\varepsilon_1,\varepsilon_2,\ldots,\varepsilon_n$ such that
$\varepsilon_i\in \{-1,1\}$ and
$y\sum\limits_{i=1}^{a} \varepsilon_i +x\sum\limits_{i=a+1}^{n} \varepsilon_i \equiv n\pmod {2n}$.
Since $n$ is divisible by $d_x$ and $x\sum\limits_{i=a+1}^{n} \varepsilon_i\equiv 0 \pmod {d_x}$,
the integer $sy$, where $s=\sum\limits_{i=1}^{a} \varepsilon_i$, has to be also divisible by $d_x$.
Notice that $s$ is odd and $-a\leq s\leq a$. Moreover, $\gcd(y,d_x)=1$, which
immediately leads to a contradiction.
The proof in the case when (2) is not satisfied follows in exactly the same way. 

To prove the sufficiency, w.l.o.g. we may assume that $a\leq n-a$ 
(otherwise it is enough to replace $y$ with $x$). Set $d=d_x$.

First, consider the case $d=1$.
Then, there exists an integer $p$ such that $1\leq p<2n$ and $x p\equiv 1\pmod{2n}$,
and there exists an integer $q$ such that $1< q < 2n$ and
$yp\equiv q \pmod{2n}$.
Set $r=\min(q, 2n-q)$; so, $r$ is odd and such that $1<  r \leq n$.
Since $\frac{n}{2}\leq n-a< n$, we can apply Lemma \ref{large_a} (when $r < n$)
or Lemma \ref{large_an} (when $r=n$, whence $d_y=n$).
Hence, there exists a perfect matching $\tilde{F}$ of $K_{2n}$ such that $\ell(\tilde{F})=\{1^{n-a},r^a\}$:
to get a perfect matching $F$ of $K_{2n}$ such that $\ell(F)=\{ x^{n-a}, y^a\}$, 
apply the relabeling $i\mapsto xi$ to all vertices of $\tilde{F}$.

It remains to consider the case $d>1$.
Then $d$ is odd, and $n=dm$, $x=dz$ for some $1\leq z < m$, where $\gcd(z,2m)=1$.
Let $\mu$ be an integer such that $1\leq \mu<2m$ and $z\mu\equiv 1\pmod{2m}$.
Let $\xi$ be a positive integer such that $1\leq \xi <2m$ and $\xi \equiv y \pmod{2m}$,
and set $\bar y =\min(\xi, 2m-\xi)$.
Let $\vartheta$ be an integer such that $0\leq \vartheta < 2m$ and $\vartheta \equiv \bar y \mu\pmod{2m}$.

Let $F'$ be a perfect matching of $K_{2m}$ such that $\ell(F')=\{z^m\}$, whose existence
follows from Proposition \ref{prop:xn}. 
We need to construct another perfect matching,
$F''$, of $K_{2m}$ such that $\ell(F'')=\{\bar y,z^{m-1}\}$.
So, working  modulo $2m$, take
$$F''  = \{\{0,\bar y\}\}\cup \left\{\{(2i-1)z,2iz\}: i\in\left[1,\frac{\vartheta-1}{2}\right]\right\}\cup \left\{\{2iz,(2i+1)z\}: i \in\left[\frac{\vartheta+1}{2},m-1\right]\right\}.$$

Let $\bar{F}=F'$ if $a$ is even and $\bar{F}=F''$ otherwise.
For each edge $\{u,w\}$ of length $z$ in $\bar{F}$ and for each $k$ such that $0\leq k\leq \frac{d-1}{2}$, we construct
a matching $A_{2k}$ of cardinality $d$ in $K_{2n}$ with $2k$ edges of length $y$ and $d-2k$ edges of length $x$.
So, working modulo $2n$, take
$$\begin{array}{rcl}
A_{2k} &=& \{\{d u+2iy,d u+(2i+1)y\},\{d w+2iy,d w+(2i+1)y\}:  i\in [0,k-1]\}\cup\\
&& \{\{d u+iy,d w+iy\}:  i\in [2k,d-1]\}.
\end{array}$$
Notice that $V(A_{2k})=\{du+iy,dw+iy: i\in[0,d-1]\}$.
Similarly, the edge $\{0,\bar y\}$ of length $\bar y$ in $\bar{F}$ corresponds to
a matching $B$ of cardinality $d$ in $K_{2n}$ such that its edges have length $y$:
$$B=\{\{2iy,(2i+1)y\}:\; i\in [0,d-1]\}$$
(also here, we  work modulo $2n$).
Notice that $V(B)=\{du+iy:\; i\in[0,2d-1]\}$.

Let $a=(d-1)b+c$ for nonnegative integers $b$ and $c$ such that $c<d-1$.
Since $a\leq \frac{n}{2}$, it follows that $(d-1)b \leq \frac{dm}{2}$, whence $b\leq\frac{3m}{4}$.
To construct a perfect matching $F$ of $K_{2n}$, we proceed separately depending on the parity of $a$.

\noindent Case I: $a$ is even.
Then $c$ is even. For $b$ edges in $\bar{F}$ we take the corresponding matching $A_{d-1}$,
and, if $c>0$, another edge is used to get $A_c$. 
For each remaining edge of $\bar{F}$  we take a matching $A_0$. 
So, we obtain $b(d-1)+c=a$ edges of length $y$ and $b+(d-c)+(m-b-1)d=n-a$ edges of length $x$.

\noindent Case II: $a$ is odd.
Then $a\geq d$, $b\geq 1$ and $c$ is odd.
A single edge of length $\bar y $ in $\bar{F}$ (the only one if $z\neq \bar y$) is used to construct $B$.
For $b-1$ edges of length $z$ in $\bar{F}$ corresponding $A_{d-1}$ are taken,
and possibly $A_{c-1}$ (if $c>1$) for one more edge of length $z$.
Each remaining edge of length $z$ in $\bar{F}$ is used to construct $A_0$.
If $c=1$ we obtain $d+(b-1)(d-1)=a$ edges of length $y$ and $(b-1)+d(m-b)=n-a$
edges of length $x$.
If $c>1$ we obtain $d+(b-1)(d-1)+(c-1)=a$ edges of length $y$ and $(b-1)+(d-c+1)+d(m-b-1)=n-a$ edges of length $x$.

In both cases,  the edges give a perfect matching $F$ of $K_{2n}$ such that $\ell(F)=\{y^a, x^{n-a}\}$.
\end{proof}

\begin{ex}
Consider the list $L=\{15^{25},35^{17}\}$. Following the notation of the proof of the previous theorem we get:
$x=15$, $y=35$, $n=42$ and $a=17$, whence $d=3$, $d_y=7$, $m=14$, $z=5$, $\bar y =7$, $\mu=17$
and $\vartheta= 7$.
First, we need to construct a perfect matching $\bar F$ of $K_{28}$ such that $\ell(\bar F)=\{7, 5^{13} \}$:
$$\begin{array}{rcl}
\bar F & =& \{\{0,7 \}\}\cup \{\{5, 10\},\{15,20\},\{25,2\} \}
\cup \{\{ 12, 17\},\{  22, 27\},\{4, 9 \},\\
&& \{ 14, 19 \},\{24, 1\},\{ 6, 11\},\{ 16, 21\},\{26, 3 \},\{ 8, 13 \},\{18, 23 \} \}.
  \end{array}$$
We now follow Case II: $b=8$ and $c=1$.
For the edge $\{0,7\}$ of $\bar F$, we construct the matching 
$B=\{\{0, 35 \}, \{ 70, 21 \}, \{ 56, 7 \}   \}$.
Then for the seven edges $\{5, 10\},\{15,20\}$, $\{25,2\}$, $\{ 12, 17\}$, $\{  22, 27\}$, 
$\{4, 9 \}$, $\{ 14, 19 \}$ of $\bar F$ we make the matchings 
$$\begin{array}{rclcrcl}
A_2^1 & = & \{\{ 15, 50\}, \{30, 65 \}, \{ 1, 16 \}\}, & \quad & 
A_2^2 & =& \{\{ 45, 80\}, \{60, 11 \},  \{ 31, 46 \}\}, \\
A_2^3 & =& \{\{  75, 26\}, \{6, 41 \}, \{ 61, 76 \}\}, && 
A_2^4 & = & \{ \{ 36, 71\}, \{51, 2 \}, \{ 22, 37 \}\},\\
A_2^5 & =& \{\{ 66, 17\}, \{ 81, 32 \}, \{ 52, 67 \}\}, && 
A_2^6 & =& \{ \{ 12, 47\}, \{27, 62 \}, \{ 82, 13 \}\},\\
A_2^7 & =& \{ \{42, 77\}, \{57, 8 \}, \{ 28, 43 \}\}, 
\end{array}$$
respectively. 
For each remaining edge of $\bar F$ we apply substitution
$A_0$ to obtain:
$$\begin{array}{rclcrcl}
A_0^1 & =& \{\{ 72, 3 \},\{ 23, 38 \},\{58, 73\}\}, & \quad & 
A_0^2 & =& \{ \{ 18, 33 \},\{ 53, 68 \},\{ 4, 19 \}\}\\
A_0^3 & = &\{\{ 48, 63 \},\{ 83, 14\},\{ 34, 49 \}\},&&
A_0^4 & =& \{\{  78, 9 \},\{ 29, 44 \},\{64, 79 \}\}, \\
A_0^5 & =& \{\{ 24, 39 \},\{ 59, 74 \},\{ 10, 25 \}\},&&
A_0^6 & =& \{\{ 54, 69 \},\{ 5, 20 \},\{ 40, 55 \}\},
  \end{array}$$
respectively.
So, $F=B\cup A_2^1\cup A_2^2\cup A_2^3\cup A_2^4\cup A_2^5\cup A_2^6\cup A_2^7\cup  
A_0^1\cup A_0^2\cup A_0^3\cup A_0^4 \cup A_0^5\cup A_0^6$.
\end{ex}

We now prove our main result.

\begin{proof}[Proof of Theorem \ref{thm:1x}]
If $d$ does not divide $n$, then there is no  perfect matching $F$ of $K_{2n}$ such that $\ell(F)=\{ x^{n-a}, y^a\}$ by Proposition \ref{prop:MPP}.
If $d=1$, the result follows from Theorems \ref{thm:xy} and \ref{thm:x odd}.
If $d=n$, then $x=y=n$, which is excluded by hypothesis. So, we may assume that $d$ is a divisor of $n$ 
such that $1<d<n$.
Let $\bar x$, $\bar y$ and $\bar n$ be integers such that $x=d\bar x$, $y=d\bar y$ and $n=d\bar n$.
Also, let $e_x=\frac{d_x}{d}$ and $e_y=\frac{d_y}{d}$.
Then, $e_x = \gcd(\bar x, 2\bar n)$ and $e_y = \gcd(\bar y, 2\bar n)$.
Note that $d_x$ divides $n$ if and only if $e_x$ divides $\bar n$.
Since $\gcd(\bar x, \bar y , 2\bar n)=1$, we have that $\bar x$ and $\bar y$ cannot be both even.
So, we can assume that $\bar y$ is odd, which implies that $d_y$ divides $n$.
It follows, by Proposition \ref{prop:xn}, that we can construct a perfect matching $\tilde F$ of $K_{2\bar n}$ such that
$\ell(\tilde F)=\{{\bar y}^{\bar n} \}$. 

By Propositions \ref{prop:b} and \ref{prop:c}, there exists a perfect matching $F$ of $K_{2n}$ such that $\ell(F)=\{ x^{n-a}, y^a\}$
if and only if 
\begin{equation}\label{F_i}
\begin{array}{l}
\textrm{there exist $d$ (not necessarily distinct) perfect matchings } F_0,F_1,\ldots, F_{d-1} \textrm{ of $K_{2\bar n}$} \\
\textrm{such that } \ell(F_i)=L_i, \textrm{ where } L_i=\{\bar x^{\bar n- a_i}, \bar y^{a_i}\} \textrm{ for suitable nonnegative integers}\\
a_0,a_1,\ldots,a_{d-1} \textrm{ such that $\sum\limits_{i=0}^{d-1} a_i=n$. }
\end{array}
\end{equation}

Suppose that $\bar x$ is even. Then, a necessary condition for the existence of the matchings $F_i$ satisfying \eqref{F_i}
is that $\bar n-a_i$ is even for all $i$. This implies that $n-a$ must be even.
So, assume that this happens. If $d_x$ divides $n$, then $\bar n$ is even and the integers $a_i$ can be chosen in the interval $[0,\bar n]$.
Take two integers $q$ and $r$ such that $n-a= 2qd + 2r $ with $0\leq r< d$, and define
$a_i=\bar n-(2q+2)$ if $0\leq i < r$ and $a_{i}=\bar n- 2q$  if $r\leq i< d$.
Notice that $0\leq 2q < \bar n$.
Since we can apply  Proposition \ref{prop:xn} and Theorem \ref{thm:xy}, condition \eqref{F_i} holds, giving case (a) of (1).
If $d_x$ does not divide $n$, then the integers $a_i$ must be chosen in  $[1,\bar n]$. 
If $2a<d_x$, then one of the integers $a_i$ must be less than $\frac{e_x}{2}$.
By Theorem \ref{thm:xy}, there is no perfect matching $F_i$ of $K_{2\bar n}$ such that 
$\ell(F_i)=\{ \bar x^{\bar n-a_i}, {\bar y}^{a_i}\}$.
Hence, we can assume $2a\geq d_x$, and take $q,r,a_i$ as before.
Note that $2a_i\geq e_x$ for every $i$, so the existence of the matchings $F_i$ satisfying \eqref{F_i} follows from Theorem \ref{thm:xy}, giving case (b) of (1).

Now, assume that $\bar x$ is odd. Notice that both $e_x$ and $e_y$ are odd.
Arguing as before, the integers $a_i$ can be chosen in $[0,\bar n]$.
If $a$ is an odd integer such that $a < e_x$, then one of
integers $a_i$ must be odd. By Theorem \ref{thm:x odd}, there is no perfect matching $F_i$ of $K_{2\bar n}$
such that  $\ell(F_i)=\{{\bar x}^{\bar n-a}, {\bar y}^{a_i} \}$. So, there is no perfect matching $F$ of 
$K_{2n}$ such that $\ell(F)=\{x^{n-a},y^a\}$.
Similarly, if $n-a$ is an odd integer such that $n-a < e_y$.
So, assume condition (3) holds.

Let $q$ and $r$ be two nonnegative integers such that $a=q\bar{n}+r$, where $r<\bar{n}$.
Note that $q<d$. Moreover, if $r$ is odd and $r<e_x$, let $s=e_x-r$, otherwise set $s=0$.
Similarly, if $\bar{n}-r$ is odd and $\bar{n}-r<e_y$, let $t=e_y-(\bar{n}-r)$, otherwise set $t=0$.
Notice that both $s$ and $t$ are even integers and at least one of them is equal to zero.
Moreover, $s<e_x\leq \frac{\bar{n}}{2}$ and $t< e_y\leq\frac{\bar{n}}{2}$.
Thus, $\bar{n}-s>e_x$ and $\bar{n}-t> e_y$. 
If $s>0$, then clearly $q\geq 1$ and  $\bar{n}-(r+s) \geq e_y$;
similarly, if $t>0$, then $q\leq d-2$ and $r-t\geq e_x$.
According to \eqref{F_i}, to obtain a perfect matching $F$ of $K_{2n}$ such that  $\ell(F)=\{ x^{n-a}, y^a\}$ it suffices to construct 
the perfect matchings $F_0,F_1,\ldots, F_{d-1}$ of $K_{2\bar n}$ such that $\ell(F_i)=L_i$, distinguishing three cases:

\noindent Case I: $s=t=0$. Take $L_i=\{\bar{y}^{\bar{n}}\}$ for $i\in[0,q-1]$, 
$L_q=\{\bar{x}^{\bar{n}-r},\bar{y}^r\}$, and $L_i=\{\bar{x}^{\bar{n}}\}$ for $i\in[q+1,d-1]$.

\noindent Case II: $s>0$. Take $L_i=\{\bar{y}^{\bar{n}}\}$ for $i\in[0,q-2]$, 
$L_{q-1}=\{\bar{x}^{s},\bar{y}^{\bar{n}-s}\}$, $L_q=\{\bar{x}^{\bar{n}-(r+s)},\bar{y}^{r+s}\}$, 
and $L_i=\{\bar{x}^{\bar{n}}\}$ for $i\in[q+1,d-1]$.

\noindent Case III: $t>0$. Take $L_i=\{\bar{y}^{\bar{n}}\}$ for $i\in[0,q-1]$,
$L_q=\{\bar{x}^{\bar{n}-(r-t)},\bar{y}^{r-t}\}$, $L_{q+1}=\{\bar{x}^{\bar{n}-t},\bar{y}^{t}\}$, and $L_i=\{\bar{x}^{\bar{n}}\}$ for $i\in[q+2,d-1]$.

In each case, the existence of the corresponding perfect matchings $F_i$ of $K_{2\bar n}$ follows from Proposition \ref{prop:xn} and Theorem \ref{thm:x odd}.
\end{proof}

\begin{ex}
Let $L=\{10^{24}, 15^6\}$. Then, $x=10$, $y=15$, $n=30$, $a=6$, $d=5$, $d_x=10$ and $d_y=15$.
Hence, $\bar x=2$, $\bar y=3$ and $\bar n=6$; so, $\bar x$ is even.
To construct a perfect matching $F$ of $K_{60}$ such that $\ell(F)=L$ we have to construct 
five perfect matchings $F_0,\ldots, F_4$ of $K_{12}$
such that $\ell(F_i)=\{2^{6-a_i}, 3^{a_i}\}$ and $a_0+\ldots+a_4=6$.
We note that $n-a=24$ is even, so we write $24 = 2q\cdot 5 + 2r$, where $q=2$ and $r=2$.
In this case $d_x$ divides $n$: so, we take $a_0=a_1=0$ and $a_2=a_3=a_4=2$.
Hence, we construct the perfect matchings $F_i$ such that 
$\ell(F_0)=\ell(F_1)=\{2^6\}$ and $\ell(F_2)=\ell(F_3)=\ell(F_4)=\{2^4, 3^2\}$.
\end{ex}

\begin{ex}
Let $L=\{10^{20}, 15^{5}\}$. Then, $x=10$, $y=15$, $n=25$, $a=5$, $d=5$, $d_x=10$ and $d_y=5$.
Hence, $\bar x=2$, $\bar y=3$ and $\bar n=5$; so, $\bar x$ is even.
We note that $n-a=20$ is even, so we write $20 = 2q\cdot 5 + 2r$, where $q=2$ and $r=0$.
In this case, $d_x$ does divides $n$ and $2a\geq d_x$: so, 
we construct the perfect matchings $F_i$ of $K_{10}$ such that 
$\ell(F_0)=\ldots=\ell(F_4)=\{2^4, 3^1\}$. Note that this can be done, since $2\cdot 1\geq \gcd(2,10)=2$.
\end{ex}

\begin{ex}
Let $L=\{75^{5}, 9^{85}\}$. Then, $x=75$, $y=9$, $n=90$, $a=85$, $d=3$, 
$d_x=15$ and $d_y=9$. Hence, $\bar x=25$, $\bar y=3$, $\bar n=30$, $e_x=5$ and  $e_y=3$.
Write $85 = q\cdot 30 + r$, where $q=2$ and $r=25$.
Since  $s=t=0$, we construct three perfect matchings $F_0,F_1,F_2$ of $K_{60}$ such that
$\ell(F_0)=\ell(F_1)=\{3^{30}\}$ and $\ell(F_2)=\{25^{5}, 3^{25}\}$.
\end{ex}

\begin{ex}
Let $L=\{70^{42}, 42^{48}\}$. Then, $x=70$, $y=42$, $n=90$, $a=48$, $d=2$, 
$d_x=10$ and $d_y=6$.
Hence, $\bar x=35$, $\bar y=21$, $\bar n=45$, $e_x=5$ and  $e_y=3$.
Write $48 = q\cdot 45 + r$, where $q=1$ and $r=3$.
Since  $s>0$, we construct two perfect matchings $F_0,F_1$ of $K_{90}$ such that
$\ell(F_0)=\{35^{2}, 21^{43}\}$  and $\ell(F_1)=\{35^{40} , 21^{5}\}$.
\end{ex}

\begin{ex}
Let $L=\{45^{16}, 25^{59}\}$. Then, $x=45$, $y=25$, $n=75$, $a=59$, $d=5$, 
$d_x=15$ and $d_y=25$. Hence, $\bar x=9$, $\bar y=5$, $\bar n=15$, $e_x=9$ and $e_y=5$.
Write $59 = q\cdot 15 + r$, where $q=3$ and $r=14$.
Since $t>0$, we construct 
five perfect matchings $F_0,\ldots,F_4$ of $K_{90}$ such that
$\ell(F_0)=\ell(F_1)=\ell(F_2)=\{5^{15}\}$, $\ell(F_3)=\{ 9^{5}, 5^{10  }\}$ and
$\ell(F_4)=\{9^{11},  5^{4}\}$.
\end{ex}

\section{Lists with all the elements with the same multiplicity}\label{sec:246}

In this section we focus on lists where each element appears the same number of times, say $t$, with underlying set
$\left\{1,2,\ldots,\frac{n}{t}\right\}$, $\left\{2,4,\ldots,\frac{2n}{t}\right\}$ or $\left\{1,3,\ldots,\frac{2n}{t}-1\right\}$.

We start by showing a connection between the problem investigated in this paper and Skolem sequences,
for details on the topic see \cite{SS}.
We recall that a \emph{Skolem sequence} of order $n$ is a sequence $S(n)=(s_0,s_1,\ldots,s_{2n-1})$ of $2n$ integers satisfying the conditions:
\begin{itemize}
\item[(1)] for every $k\in [1,n]$ there exist two elements $s_i,s_j\in S(n)$ such that $s_i=s_j=k$,
\item[(2)] if $s_i=s_j=k$ with $i<j$, then $j-i=k$.
\end{itemize}
Skolem sequences are also written as collections of ordered pairs $\{(a_i,b_i): 1\leq i \leq n, b_i-a_i=i\}$
with $\cup_{i=1}^n\{a_i,b_i\}=[0,2n-1]$, which can be seen as the edges of a perfect matching $F$ of $K_{2n}$.
It is easy to see that the list of edge-lengths of $F$ is nothing but the set  $[1, n]$.
For instance, the Skolem sequence  $S(5)=(1,1,3,4,5,3,2,4,2,5)$  of order $5$ can be seen as the perfect matching
$F=\{\{0,1\},\{6,8\},\{2,5\},\{3,7\},\{4,9\}\}$ of $K_{10}$ such that $\ell(F)=\ell'(F)=[1,5]$.

\begin{prop}\label{prop:Skolem}
Let $L=\{1,2,\ldots,n\}$. There exists a perfect matching $F$ of $K_{2n}$ such that $\ell(F)=L$
  if and only if $n\equiv 0,1\pmod 4$.
	Moreover, $\ell(F)=\ell'(F)$ holds.
\end{prop}

\begin{proof}
For every positive integer $n\equiv 0,1\pmod 4$, there exists a Skolem sequence of order $n$, see \cite{SS},
hence the result follows by previous considerations.
If $n\equiv 2,3\pmod 4$, $L$ contains an odd number of even numbers, hence the non-existence of $F$ follows by Proposition \ref{prop:EvenNumber}.
\end{proof}

\begin{cor}
Given $t\leq n$, let $L=\{i^{a_i} : i \in [1,t]\}$ be such that $|L|=n$, $a_i\geq a_{i+1}\geq 1$ and $a_{4k+2}=a_{4k+3}=a_{4k+4}$ for any $k$.
Then there exists a perfect matching $F$ of $K_{2n}$ such that $\ell(F)=\ell'(F)=L$.
\end{cor}
\begin{proof}
Note that $a_{4k+2}=a_{4k+3}=a_{4k+4}$  implies $t\equiv 0,1 \pmod 4$.
The result follows by Proposition \ref{prop:Skolem} and Remark \ref{rem:F1+F2}.
\end{proof}

\begin{ex}
Let $L=\{1^5,2^4,3^4,4^4,5^4,6^2,7^2,8^2,9\}$, so $|L|=28$.
Let $S_i$ be a Skolem sequence of order $i$, with $i\equiv 0,1\pmod 4$,  and consider the corresponding perfect matching $F_i$.
Then $$F=F_9\cup(F_8+18)\cup (F_5+34)\cup (F_5+44) \cup (F_1+54)$$
is a perfect matching of $K_{56}$ such that  $\ell(F)=\ell'(F)=L$.
\end{ex}

To get a generalization of Proposition \ref{prop:Skolem} we consider the case in which all the elements in the list are odd (even, respectively)
integers.

\begin{prop}\label{prop:odd}
Let $t\geq 2$.
There exists a perfect matching $F$ of $K_{2n}$ such that $\ell(F)=\left\{1^t,3^t,\ldots,(\frac{2n}{t}-1)^t\right\}$
if and only if $n\equiv 0 \pmod t.$
\end{prop}
\begin{proof}
It is trivial that $n\equiv 0 \pmod t$ is a necessary condition.
On the other hand, one can easily check that
$$F=\left\{\left\{ \frac{2n}{t}i+j, \frac{2n}{t}(i+1)-j-1 \right\} : i\in[0,t-1], j\in\left[0,\frac{n}{t}-1\right]\right\}$$
satisfies the required conditions. Note that also in this case $\ell(F)=\ell'(F)$.
\end{proof}

\begin{ex}\label{ex:odd}
A perfect matching $F$ of $K_{32}$ such that $\ell(F)=\ell'(F)=\{1^4,3^4,5^4,7^4\}$ is
$$\begin{array}{rcl}
F&=&\{\{0,7\},\{1,6\},\{2,5\},\{3,4\},\{8,15\},\{9,14\},\{10,13\},\{11,12\},\{16,23\},\{17,22\},\\
&&\{18,21\},\{19,20\},\{24,31\},\{25,30\},\{26,29\},\{27,28\}\}.
\end{array}$$
\end{ex}

\begin{prop}\label{prop:even}
Let $t\geq 2$. There exists a perfect matching $F$ of $K_{2n}$ such that $\ell(F)=\left\{2^t,4^t,\ldots, \left(\frac{2n}{t}\right)^t\right\}$
if and only if $n\equiv 0 \pmod s$, where $s=t$ if $t$ is even and $s=4t$ otherwise.
\end{prop}

\begin{proof}
To prove necessity, first of all we need to realize that $n\equiv 0 \pmod t$ is a trivial condition.
Moreover, since $\frac{2n}{t}\leq n$ then $t\geq 2$.
By Proposition \ref{prop:EvenNumber}, $n$ is even. Suppose to the contrary that
$t$ is odd and $n\equiv 2t \pmod {4t}$. The existence of $F$ would imply, by Proposition \ref{prop:c},
 the existence of two perfect matchings $F_0,F_1$
of $K_{n}$, one of which has the list of edge-lengths containing an odd number of even integers.
Hence, by Proposition \ref{prop:EvenNumber}, we get a contradiction.

Let $t$ be even. Then it is easily seen that
$$\begin{array}{rcl}
F_{n,t} & =& \left\{\left\{\frac{(2i+1)n}{t}-j,\frac{(2i+1)n}{t}+j \right\}:\; i\in [0,t-1], j\in \left[1,\frac{n}{t}-1\right]\right\}\cup\\[4pt]
&& \left\{\left\{\frac{4in}{t},\frac{(4i+2)n}{t}\right\},\left\{\frac{(4i+1)n}{t},\frac{(4i+3)n}{t}\right\}:\; i\in \left[0,\frac{t}{2}-1\right]\right\}
  \end{array}$$
is a perfect matching of $K_{2n}$ with the required properties. Notice that $\ell(e)=\ell'(e)$ for each $e\in F_{n,t}$.

Let $t$ be odd. Since $t\geq 3$ we have $n\geq 12$. Firstly, we construct a perfect matching $F'$ of $K_{24}$
such that $\ell(F')=\ell'(F')=\{2^3,4^3,6^3,8^3\}$.
Namely,
$F'=\{\{0,8\},\{1,7\},\{2,6\},\{3,5\},$ $\{4,12\},\{9,15\},\{10,16\},\{11,13\},\{14,22\},\{17,21\},\{18,20\},\{19,23\}\}$.
Now, set $F''=F'$ if $t=3$,
$F''=F'\cup (F_{4(t-3),t-3}+24)$ if $t\geq 5$. In both cases $F''$ is a  matching of $K_{2n}$
such that $V(F'')=[0,8t-1]$ and  $\ell(F'')=\{2^t,4^t,6^t,8^t\}$.
If $n=4t$ we have done. Otherwise,
on the set $V(F'')$ apply the relabeling $i\mapsto \frac{ni}{4t}$.
In this way $F''$ is converted into a matching $F^2$ of $K_{2n}$ such that
$\ell(F^2)=\left\{(\frac{n}{2t})^t,(\frac{n}{t})^t,(\frac{3n}{2t})^t,(\frac{2n}{t})^t\right\}$ and
$V(F^2)=U$,
where $U=\left\{0,\frac{n}{4t},\frac{n}{2t},\ldots,2n-\frac{n}{4t}\right\}$. Note that
 $|U|=8t$. Set $A=\left\{\frac{n}{4t},\frac{n}{2t},\frac{3n}{4t},\frac{n}{t}\right\}$
and
$$F^1=\left\{\left\{\frac{(2i+1)n}{t}-j,\frac{(2i+1)n}{t}+j\right\}:\; i\in [0,t-1],\, j\in \left[1,\frac{n}{t}\right]\setminus A\right\}.$$
It is easy to see that $F^1$ is a matching of $K_{2n}$ such that $|F^1|=n-4t$, $V(F^1)=V(K_{2n})\setminus U$ and
$\ell(F^1)=\{(2j)^t : j \in\left[1,\frac{n}{t}\right]\setminus A\}$.  
Thus $F^1\cup F^2$ is a perfect matching of $K_{2n}$ with the required properties. Moreover,
$\ell(e)=\ell'(e)$ for each $e\in F^1\cup F^2$.
\end{proof}

\begin{ex}\label{ex:even}
A perfect matching $F$ of $K_{24}$ such that $\ell(F)=\ell'(F)=\{2^4,4^4,6^4\}$ is
$$\begin{array}{rcl}
F=F_{12,4} &=&\{\{2,4\},\{1,5\},\{8,10\},\{7,11\},\{14,16\},\{13,17\},\{20,22\},\{19,23\}\}\cup\\
&&\{\{0,6\},\{3,9\},\{12,18\},\{15,21\}\}.
\end{array}$$
\end{ex}

\begin{ex}
Here, we consider the list $\{2^7,4^7,6^7,8^7,10^7,12^7,14^7,16^7\}$, so we are working in $K_{112}$.
Firstly, we construct $F''=F'\cup (F_{16,4}+24)$:
$$\begin{array}{rcl}
F''&=&\{\{0,8\},\{1,7\},\{2,6\},\{3,5\},\{4,12\},\{9,15\},\{10,16\},\{11,13\},\{14,22\},\\
&& \{17,21\},\{18,20\},\{19,23\},\{27,29\},\{26,30\},\{25,31\},\{35,37\},\{34,38\},\\
&&\{33,39\},\{43,45\},\{42,46\},\{41,47\},\{51,53\},\{50,54\},\{49,55\},\{24,32\},\\
&& \{40,48\},\{28,36\},\{44,52\}\}.
\end{array}$$
Note that $\ell(F'')=\{2^7,4^7,6^7,8^7\}$.
Now, we construct $F^2$ by applying the relabeling $i\mapsto 2i$:
$$\begin{array}{rcl}
F^2&=&\{\{0,16\},\{2,14\},\{4,12\},\{6,10\},\{8,24\},\{18,30\},\{20,32\},\{22,26\},\\
&& \{28,44\},\{34,42\},\{36,40\},\{38,46\},\{54,58\},\{52,60\},\{50,62\},\{70,74\},\\
&& \{68,76\},\{66,78\},\{86,90\},\{84,92\}, \{82,94\},\{102,106\},\{100,108\},\\
&& \{98,110\},\{48,64\},\{80,96\},\{56,72\},\{88,104\}\}.
\end{array}$$
Clearly, $\ell(F^2)=\{4^7,8^7,12^7,16^7\}$.
Finally we construct $F^1$:
$$\begin{array}{rcl}
F^1&=&\{\{7,9\},\{5,11\},\{3,13\},\{1,15\},\{23,25\},\{21,27\},\{19,29\},\{17,31\},\\
&&\{39,41\},\{37,43\},\{35,45\},\{33,47\},\{55,57\},\{53,59\},\{51,61\},\{49,63\},\\
&&\{71,73\},\{69,75\},\{67,77\},\{65,79\},\{87,89\},\{85,91\},\{83,93\},\{81,95\},\\
&&\{103,105\},\{101,107\},\{99,109\},\{97,111\}\}.
\end{array}$$
It results $\ell(F^1)=\{2^7,6^7,10^7,14^7\}$.
Take $F=F^1\cup F^2$. 
\end{ex}

\begin{cor}\label{1234}
Let $n$, $t$ be two integers such that $t$ divides $n$.
There exists a perfect matching $F$ of $K_{2n}$ such that $\ell(F)=\left\{1^t,2^t,\ldots, \left(\frac{n}{t}\right)^t\right\}$
if and only if either $t$ is even or $\frac{n}{t}\equiv 0,1\pmod 4$.
\end{cor}

\begin{proof}
If $t$ is odd and $\frac{n}{t}\equiv 2,3\pmod 4$ the non-existence follows by Proposition \ref{prop:EvenNumber}.
If $\frac{n}{t}\equiv 0,1 \pmod 4$, we apply Remark \ref{rem:F1+F2} and Proposition \ref{prop:Skolem}.
If $\frac{n}{t}\equiv 2,3 \pmod 4$ and $t$ is even we apply Remark \ref{rem:F1+F2} and Propositions \ref{prop:odd} and \ref{prop:even}.
\end{proof}

\begin{ex}
Let $F$ and $F'$ be the matchings constructed in Examples \ref{ex:odd} and \ref{ex:even}, respectively.
Then, $F''=F\cup(F'+32)$,
namely
$$\begin{array}{rcl}
F''&=&\{\{0,7\},\{1,6\},\{2,5\},\{3,4\},\{8,15\},\{9,14\},\{10,13\},\{11,12\},\{16,23\},\{17,22\},\\
&&\{18,21\},\{19,20\},\{24,31\},\{25,30\},\{26,29\},\{27,28\}\}\cup\\
&&\{\{34,36\},\{33,37\},\{40,42\},\{39,43\},\{46,48\},\{45,49\},\{52,54\},\{51,55\},\\
&&\{32,38\},\{35,41\},\{44,50\},\{47,53\}\}
\end{array}$$
is a perfect matching of $K_{56}$ such that $\ell(F'')=\ell'(F'')=\{1^4,2^4,3^4,4^4,5^4,6^4,7^4\}$.
\end{ex}

We conclude this section considering some similar lists.

\begin{prop} \label{prop:4cases}
Let $n\geq 3$ be an odd integer. There exists a perfect matching $F$ of $K_{2n}$ such that:
\begin{itemize}
\item[(1)] $\ell(F)=\{1^2,3^2,\ldots, (n-2)^2,n\}$;
\item[(2)] $\ell(F)=\{2^2,4^2,\ldots, (n-1)^2,n\}$.
\end{itemize}
\end{prop}
\begin{proof}
In the case (1) take $F=\{\{i,2n-1-i\}: i\in[0,n-1]\}$,
while in the case (2) take
$$F=\left\{\{i,n-1-i\}: i\in \left[0,\frac{n-3}{2}\right] \cup \left[n,\frac{3n-3}{2}\right]\right\} \cup
\left\{\left\{\frac{n-1}{2},\frac{3n-1}{2}\right\}\right\}.$$
One can check that, in both  cases, $F$ satisfies the required properties.
\end{proof}

\section{Conclusions and open problems}

The conditions presented in this paper lead us to believe that it is not possible
to find a ``nice'' statement for the seating couple problem in the even case, as done for the odd case,
namely to find a condition such as \eqref{PP} of Conjecture \ref{MPP}.
In fact, as we have seen, the necessary conditions given in Section \ref{sec:preliminary}
are rarely sufficient. Hence, it would be interesting to classify the lists for which it happens.
This is the case when the underlying set has length $1$ (Proposition \ref{prop:xn})
or consists of the consecutive integers $1,2,\ldots,x$,
each appearing in the list with the same multiplicity (Corollary \ref{1234}).
It is also the case described in Theorem~\ref{thm:xy}.

On the other hand, one could start considering the case where $n$ is an odd prime, as done in \cite{M}.
With this assumption, when the list does not contain $n$,  the necessary conditions of Theorem \ref{thm:1x} simply become those of Propositions \ref{prop:MPP} and \ref{prop:EvenNumber}. So, also in view of some computational results, we propose the following
conjecture which is clearly related to Theorem~\ref{thm:KSC}, where the elements of the list are all coprime with $2n$.
Here, we assume the stronger assumption that $n=p$ is an odd prime, but the list is allowed to contain 
also even integers.

\begin{conj}\label{conj:3}
Let $p$ be an odd prime and let $L$ be a list of $p$ positive integers less than $p$.
There exists a perfect matching $F$ of $K_{2p}$ such that $\ell(F) = L$ if and only if the number of even
integers in $L$ is even.
\end{conj}

\section*{Acknowledgements}
The second and the third author are partially supported by INdAM-GNSAGA.

\end{document}